\documentclass[11pt]{amsart}

\usepackage{fullpage}

\usepackage{amsmath, amscd, amsthm, amssymb}

\usepackage{hyperref}

\def\C{{\mathbf C}}
\def\R{{\mathbf R}}

\def\Z{{\mathbf Z}}
\def\Q{{\mathbf Q}}
\def\A{{\mathbf A}}

\def\H{{\mathbf H}}

\newtheorem{theorem}{Theorem}[subsection]

\newtheorem{lemma}[theorem]{Lemma}

\newtheorem{proposition}[theorem]{Proposition}

\newtheorem{corollary}[theorem]{Corollary}

\theoremstyle{definition}

\theoremstyle{remark}

\newcommand{\mm}[4]{\left(\begin{smallmatrix} #1 & #2\\ #3 & #4\end{smallmatrix}\right)}
\newcommand{\mb}[4]{\left(\begin{array}{cc} #1 & #2\\ #3 & #4\end{array}\right)}

\DeclareMathOperator{\tr}{tr}

\DeclareMathOperator{\SO}{SO}
\DeclareMathOperator{\Spin}{Spin}

\DeclareMathOperator{\Sp}{Sp}

\DeclareMathOperator{\PGSp}{PGSp}

\DeclareMathOperator{\SU}{SU}

\DeclareMathOperator{\SL}{SL}
\DeclareMathOperator{\GL}{GL}

\DeclareMathOperator{\diag}{diag}

\DeclareMathOperator{\Span}{Span}

\def\p{{\mathfrak p}}

\def\so{{\mathfrak {so}}}
\def\sl{{\mathfrak {sl}}}

\def\Vm{{\mathbb{V}}}

\def\H{{\mathcal{H}}}

\def\D{{\mathcal D}}

\def\W{{\mathbf W}}
\def\X{{\mathbf X}}
\def\Y{{\mathbf Y}}

\begin{document}
\title{A quaternionic Saito-Kurokawa lift and cusp forms on $G_2$}
\author{Aaron Pollack}
\address{Department of Mathematics\\ Duke University\\ Durham, NC USA}
\email{apollack@math.duke.edu}
\thanks{The author has been supported by the Simons Foundation via Collaboration Grant number 585147.}

\begin{abstract} We consider a special theta lift $\theta(f)$ from cuspidal Siegel modular forms $f$ on $\Sp_4$ to ``modular forms" $\theta(f)$ on $\SO(4,4)$, in the sense of \cite{pollackQDS}.  This lift can be considered an analogue of the Saito-Kurokawa lift, where now the image of the lift is representations of $\SO(4,4)$ that are quaternionic at infinity.  We relate the Fourier coefficients of $\theta(f)$ to those of $f$, and in particular prove that $\theta(f)$ is nonzero and has algebraic Fourier coefficients if $f$ does.  Restricting the $\theta(f)$ to $G_2 \subseteq \SO(4,4)$, we obtain cuspidal modular forms on $G_2$ of arbitrarily large weight with all algebraic Fourier coefficients.  In the case of level one, we obtain precise formulas for the Fourier coefficients of $\theta(f)$ in terms of those of $f$.  In particular, we construct nonzero cuspidal modular forms on $G_2$ of level one with all integer Fourier coefficients.
\end{abstract}
\maketitle


\section{Introduction} Recall the notion of ``modular forms" on $G_2$ from \cite{ganGrossSavin}.  These are elements of the space $Hom_{G_2(\R)}(\pi_n,\mathcal{A}(G_2))$ of $G_2(\R)$-equivariant homomorphisms from the quaternionic discrete series $\pi_n$ \cite{grossWallach1,grossWallach2} to the space of automorphic forms on $G_2(\A)$.  Equivalently, see \cite{pollackQDS} or \cite{pollackG2}, they are certain $Sym^{2n}(\C^2)$-valued automorphic functions on $G_2(\A)$ that are annihilated by a special linear differential operator $\D_n$. In \cite{ganGrossSavin}, Gan-Gross-Savin developed the theory of the (non-degenerate) Fourier coefficients of modular forms on $G_2$ using as a key input a certain Archimedean multiplicity one result of Wallach \cite{wallach}, (see also \cite[section 15]{ganSW}.)  The full Fourier expansion, including the degenerate terms, of modular forms on quaterionic exceptional groups was then developed in \cite{pollackQDS}.  See also \cite{pollackE8}.

While the general theory of the Fourier expansion is now worked out for these modular forms on quaternionic exceptional groups, there is currently a short supply of concrete examples.  Specifically, only in \cite{ganGrossSavin} and in \cite{pollackE8} are examples given which provably have relatively nice Fourier expansions, in the sense that it is proved that most of the Fourier coefficients are algebraic numbers.  In particular, no examples have been given of \emph{cusp forms} that have all algebraic Fourier coefficients--or, for that matter, any explicit examples of nonzero cusp forms.  Thus, it is natural to ask if there exist cuspidal modular forms on $G_2$ all of whose Fourier coefficients are algebraic.  Our first theorem settles this question in the affirmative.

\begin{theorem} There are nonzero cuspidal modular forms on $G_2$ of arbitrarily large weight, all of whose Fourier coefficients are algebraic numbers. \end{theorem}

To prove this theorem, we develop an analogue of the Saito-Kurokawa lift, or more generally, the Oda-Rallis-Schiffman lift \cite{oda}, \cite{RS78,RS81}, see also \cite{kudla}.   Recall that the Saito-Kurokawa lift can be considered as a very special case of the $\theta$-lift from holomorphic modular forms on $\widetilde{\SL_2}$ to holomorphic Siegel modular forms on $\SO(5) = \PGSp_4$, and the Fourier coefficients of the lift can be neatly described in terms of those of the input on $\widetilde{\SL_2}$.  These types of special $\theta$ lifts, in turn, go back to Doi-Naganuma \cite{doiNaganuma}, Niwa \cite{niwa} and Shintani \cite{shintani}.  

The lifts we consider now start off with cuspidal Siegel modular forms $f$ on $\Sp_4$, and via very special test data for the Weil-representation, we lift them to automorphic forms on $\SO(4,4)$.  The lifts are cuspidal by a general argument of Rallis \cite[Chapter I, section 3]{rallisHD}.  With our special test data, we are able to check that the lifts are nonzero (quaternionic) modular forms in the sense of \cite{weissman} and \cite{pollackQDS}.  Moreover, we can prove that the lift to $\SO(4,4)$ preserves algebraicity in the sense that if $f$ has Fourier coefficients in some field $E$ containing the cyclotomic extension of $\Q$, then $\theta(f)$ also has Fourier coefficients in $E$.  Except for the cuspidality on $\SO(4,4)$, these properties of the theta lift partially generalize to a special quaternionic $\theta$-lift from cuspidal Siegel modular forms on $\Sp_4$ to quaternionic modular forms on $\SO(4,n)$.  Let us remark that this special $\theta$-lift from $\Sp(4)$ to $\SO(4,n)$ is also inspired by \cite{narita1} (following unpublished work of Arakawa) who lifts cuspidal modular forms on $\SL_2$ to quaternionic modular forms on $\Sp(1,q)$, which sits inside $\SO(4,4q)$.  

Continuing with the dual pair $\Sp_4 \times \SO(4,4)$, one can restrict automorphic functions on $\SO(4,4)$ to $G_2$. Perhaps surprisingly, the restriction of cuspidal nonzero modular forms on $\SO(4,4)$ to $G_2$ remains a \emph{cuspidal} nonzero modular form.  Moreover, the algebraicity of the Fourier coefficients is preserved under restriction to $G_2$.  Because the cuspidal Hecke eigenforms on $\Sp_4$ always have Fourier coefficients in a finite extension of $\Q$, the above procedure produces cuspidal modular forms on $G_2$ of arbitrarily large weight, with all algebraic Fourier coefficients.

The method of $\theta$-lifting to an orthogonal group and then restricting to $G_2$ comes from Rallis-Schiffmann \cite{RSG2}.  There, the authors lift from $\widetilde{SL_2}$ to split $\SO(7)$, and then restrict to $G_2$.  These Rallis-Schiffmann lifts to $G_2$ do not produce \emph{modular forms} on $G_2$.  Indeed, Li and Schwermer \cite[section 3.8]{liSchwermer} compute that the discrete series on $G_2(\R)$ that are in the image of the Rallis-Schiffmann lift are the ones that have minimal $K$-type nontrivial representations of the short-root $\SU(2)$, as opposed to the long-root $\SU(2)$. In particular, we know of no good way to measure the algebraicity of the Rallis-Schiffmann lifts, because they are not modular forms. See also \cite{naritaYamauchi}, where the authors construct special automorphic functions on $\SO(4,1)$ with algebraic Fourier coefficients, by restricting theta functions from $\SO(4,2)$.

In the case when the input $F(Z)$ is a \emph{level one} cuspidal holomorphic modular form on $\Sp(4)$, we obtain precise formulas for the Fourier coefficients $\theta(F)$ on $\SO(V)$ and $\theta(F)|_{G_2}$ on $G_2$.  See Theorem \ref{thm:level1} and Corollary \ref{cor:G2level1}.  As a consequence of these results, one has the following.

\begin{theorem} Suppose $F(Z)= \sum_{T > 0}{a_F(T) q^T}$ is a level one cuspidal holomorphic modular form on $\Sp(4)$ of sufficiently large weight, with Fourier coefficients $a_F(T)$ in some ring $R$.  Assume moreover that the Fourier coefficient $a_F(\mm{1}{0}{0}{1}) \neq 0$.  Then $\theta(F)|_{G_2}$ is a nonzero cuspidal modular form on $G_2$ with Fourier coefficients in $R$.\end{theorem}
The theorem produces cuspidal modular forms on $G_2$ with all integral Fourier coefficients.

\noindent \textbf{Acknowledgements}  We thank Hiro-aki Narita for helpful conversations.

\section{Generalities on the theta lift}
In this section, we describe those results on the $\theta$-lift from $\Sp_4$ to $\SO(4,n)$ that do not depend on our specific test data or the notion of modular forms on quaternionic groups. 

\subsection{Weil representation}\label{subsec:Weil} We discuss various notations, definitions, and recall well-known facts for the Weil representation \cite{weil} restricted to $\SO(V) \times \Sp(W)$ when $\dim(V)=m$ is even.  See also, e.g., \cite{kudlaInv},\cite{rallisHD, rallisLF}.

Thus suppose $k$ is a local or global field of characteristic $0$.  Let $(V,q)$ be a non-degenerate quadratic space over $k$, and let $(x,y) = q(x+y)-q(x)-q(y)$ be the associated symmetric bilinear form, so that $q(x) = \frac{1}{2}(x,x)$.  Set $\det(V) \in k^\times/(k^\times)^2$ the element $\det((v_i,v_j)_{i,j})$ where $v_1, \ldots, v_m$ is a basis of $V$.  Set $\mathrm{disc}(V) = (-1)^{m(m-1)/2}\det(V) \in k^\times/(k^\times)^2$ the discriminant of $V$.  When $k$ is a local field, let $(a,b) \in \mu_2$ be the Hilbert symbol of $k$ and denote by $\chi_V: k^\times \rightarrow \mu_2$ the character $\chi_V(x) = (x,\mathrm{disc}(V))$.  When $k$ is a global field, write $\chi_V$ for the quadratic character of $\A_k$ whose local components are the $(\cdot,\mathrm{disc}(V))$ just defined. 

Let $\psi: \A/k \rightarrow \C^\times$ be a fixed additive character.  Below when $k=\Q$, we will take $\psi$ to be the standard choice, so that $\psi_\infty(x) = e^{2\pi i x}$.  For $f$ a Schwartz-Bruhat function, we write
\[\widehat{f}(x) = \int_{V}{\psi((y,x))f(y)\,dy}\]
the Fourier transform of $f$.  The measure $dy$ is normalized so that $\widehat{\widehat{f}}(x) = f(-x)$.

For $\rho \in k^\times$, define the Weil index $\gamma(\psi,\rho q(\cdot))$ the element of $\C^\times$ so that
\[
\int_{V}{\psi(\rho q(x)) \widehat{f}(x)\,dx} = \gamma(\psi, \rho q(\cdot)) |\rho|^{-m/2} \int_{V}{\psi(- \rho^{-1}q(x))f(x)\,dx}\]
for all Schwartz-Bruhat functions $f$.

Now suppose $W$ is a symplectic space over $k$, and $W = X \oplus Y$ is a Langrangian decomposition.  Moreover, assume given a symplectic basis $e_1, \ldots, e_n, f_1, \ldots, f_n$ for $W$, so that $X = \Span\{e_1, \ldots, e_n\}$ and $Y = \Span\{f_1,\ldots, f_n\}$.   We write $\X  = X\otimes V = V^n$ and $\Y = Y \otimes V = V^n$.  When $m = \dim(V)$ is even, we have a Weil representation $\omega_\psi$ of $\SO(V) \times \Sp(W)$ on the Schwartz space $S(\X)$ locally and globally.  

We let $\SO(V)$ and $\Sp(W)$ act on the right of $V$, resp. $W$.  Then if $g \in \SO(V)$ and $\phi \in S(X)$ one has $\omega_\psi(g,1)\phi(x) = \phi(xg)$.  The Weil representation restricted to $\Sp(W)$ is the unique representation for which 
\begin{equation}\label{eqn:Paction} \mb{\alpha}{\beta}{}{\,^t\alpha^{-1}}\phi(x) = \chi_V(\det(\alpha)) |\det(\alpha)|^{m/2}\psi(\langle x\alpha, x\beta \rangle/2) \phi(x\alpha)\end{equation}
and
\[\mb{0}{1}{-1}{0}\phi(x) = \gamma(\psi,q)^{n} \widehat{\phi}(x)\]
where
\[\widehat{\phi}(x) = \int_{V^n}{\psi(\tr((x_i,y_j)))\phi(e_1\otimes y_1 + \cdots + e_n \otimes y_n)\,dy_1 \cdots \,dy_n}.\]

We will require the use of a partial Fourier transform to go between different models of this Weil representation.  We explain this now.  Suppose $V = U \oplus V_0 \oplus U^\vee$ with $U, U^\vee$ isotropic, paired nontrivially via the symmetric form, and $V_0$ non-degenerate, $V_0 = (U\oplus U^\vee)^\perp$.  Note that $\mathrm{disc}(V_0) = \mathrm{disc}(V)$ so that $\chi_{V} = \chi_{V_0}$.  With $\X_U := X \otimes V_0 \oplus W \otimes U$ and $\Y_U := Y \otimes V_0 \oplus W \otimes U^\vee$ one has that $\W = W \otimes V = \X_U \oplus \Y_U$ is a Lagrangian decomposition.

There is a different model of $\omega_\psi$, now on $S(\X_U)$.  The intertwining operator between these two models is given by a partial Fourier transform.  This transform is defined as follows.  First, we have
\[
\X = (U \oplus V_0 \oplus U^\vee) \otimes X = U \otimes X \oplus V_0 \otimes X \oplus U^\vee \otimes X
\]
and
\[
\X_U = U \otimes W \oplus V_0 \otimes X = U \otimes X \oplus V_0 \otimes X \oplus U \otimes Y.
\]
Then if $\phi \in S(\X)$,
\[
F(\phi)(x,y,z) = \int_{(U^\vee \otimes X)(\A)}{\psi(\langle z', z\rangle)\phi(x,y,z')\,dz'}
\]
where $x \in U \otimes X$, $y \in V_0 \otimes X$, $z \in U \otimes Y$ and $z' \in U^\vee \otimes X$.

The partial Fourier transform defines an isomorphism $S(\X) \rightarrow S(\X_U)$.  Thus by transport of structure, there is Weil representation $\omega_{\psi,\Y_U}$ of $\SO(V) \times \Sp(W)$ on $S(\X_U) = S(X \otimes V_0) \otimes S(W \otimes U)$.  This representation has the following well-known and useful property.  Denote by $P_U$ the parabolic subgroup of $\Sp(\W)$ that stabilizes $\Y_U$.  For $p \in P_U$, let $\det_{P_U}(p) = \det(p: \W/\Y_U \rightarrow \W/\Y_U)$ be the determinant of the map induced by $p$ on $\W/\Y_U$.  Finally, set $p_{\X_U}$ resp $p_{\Y_U}$ the projections $\W \rightarrow \X_U$ and $\W \rightarrow \Y_U$.

\begin{proposition} Assume that $\mathrm{disc}(V) = 1$.  Suppose $p \in P_U \cap \left(\SO(V) \times \Sp(W)\right)$ and $\phi \in S(\X_U)$. Then
\begin{equation}\label{PXU}
\omega_{\psi,\Y_U}(p)\phi(x) = |\det_{P_U}(p)|^{1/2} \psi(\langle p_{\X_U}(xp), p_{\Y_U}(yp)\rangle/2) \phi(p_{\X_U}(xp)).\end{equation}
In particular, $\Sp(W)$ acts linearly on $S(W \otimes U)$.
\end{proposition}
\begin{proof} As mentioned, this is well-known. See, for example, Kudla's notes ``The local theta correspondence", Lemma 4.2 or \cite[pg 340-341]{rallisHD}. Because one already knows that the Weil representation exists on $S(\X)$, the proposition can be checked by assuming $\phi = F(\phi')$ for $\phi' \in S(\X)$ and checking \eqref{PXU} for generators of $\Sp(W)$ and $\SO(V)$.\end{proof}

If $\phi \in S(\X(\A))$, the $\theta$-function associated to $\phi$ is
\[\theta(g,h;\phi) = \sum_{\xi \in \X(k)}{\omega_{\psi}(g,h)\phi(\xi)}.\]
Here $g \in \SO(V)(\A)$ and $h \in \Sp(W)(\A)$.  The function $\theta(g,h;\phi)$ is an automorphic form on $\SO(V)(\A) \times \Sp(W)(\A)$.

If $\phi' = F(\phi) \in S(\X_U)$, then one can similarly define
\[\theta(g,h;\phi') = \sum_{\xi \in \X_U(k)}{\omega_{\psi,\Y_U}(g,h)\phi'(\xi)}.\]
\begin{proposition}\label{prop:PFglobal} Let the notation be as above, so that $\phi'=F(\phi)$.  Then $\theta(g,h;\phi') = \theta(g,h;\phi)$.\end{proposition}
\begin{proof} This follows from Poisson summation, and is well-known.\end{proof}

Finally, given a cuspidal automorphic form $f$ on $\Sp(W)$ and $\phi \in S(\X(\A))$, the theta-lift of $f$ is
\[\theta(f,\phi)(g) = \int_{\Sp(W)(k)\backslash \Sp(W)(\A)}{\theta(g,h;\phi)f(h)\,dh}.\]
By Proposition \ref{prop:PFglobal},
\[\theta(f,\phi)(g) = \int_{\Sp(W)(k)\backslash \Sp(W)(\A)}{\theta(g,h;F(\phi))f(h)\,dh}.\]

\subsection{Definitions and notation} We now give the specific notations that we will use below. 

Throughout the paper $n = 8k+4$ for a non-negative integer $k$.  Moreover, $V$ is an $8k+8$-dimensional $\Q$-vector space that has a quadratic form $q$ of signature $(4,n) = (4,8k+4)$.  Inside of $V$, $\Lambda \subseteq V$ is an even unimodular lattice for the quadratic form $q$. Thus we assume that the discriminant of $V$ is trivial.  The symmetric bilinear form associated to $q$ is $(x,y) = q(x+y)-q(x)-q(y)$ for $x,y \in V$.

We fix a decomposition
\[ 
V = U \oplus V_0 \oplus U^\vee
\]
as above where $U, U^\vee$ are isotropic and two-dimensional, and duals to each other under the symmetric pairing on $V$.  Here $V_0 = (U \oplus U^\vee)^\perp$ is an orthogonal space of signature $(2,n-2)$, and we denote by $q_0$ the restriction of $q$ to $V_0$.  We write $b_1, b_2$ for a fixed basis of $U$, and $b_{-1}, b_{-2}$ for the dual basis of $U^\vee$.

Over $\R$, we fix an orthogonal decomposition $V_0 \otimes \R = \C \oplus V_+$ so that $q_0(z,h) = |z|^2 - q_+(h)$, where $(V_+, q_+)$ is a positive-definite orthogonal space.  With this decomposition fixed, we have $q((x,(z,h),\delta)) = \delta(x) + |z|^2 - q_+(h)$, where $x \in U$, $\delta \in U^\vee$ and $z,h$ as above.  

We obtain a majorant of $q$ as follows.  Define $\iota: U \rightarrow U^\vee$ via $\iota(b_j) = b_{-j}$ for $j = 1,2$.  Extend $\iota$ to $V$ by defining $\iota(z,h) = (z,-h)$ for $(z,h) \in V_0$.  Then
\[ 
((x,z,h,\delta),\iota(x,z,h,\delta)) = (x,\iota(x)) + (\iota(\delta),\delta) + 2|z|^2 + 2 q_+(h) \geq 0.
\]
We set $r(v) = ||v||^2 = \frac{1}{2} (v,\iota(v))$ the majorant.

Denote by $W$ the defining four-dimensional representation of $\Sp(4) = \Sp(W)$.  We fix a polarization $W = X \oplus Y$ of Lagrangian subspaces.  Write $e_1, e_2, f_1, f_2$ for a fixed symplectic basis of $W$, so that $X$ is spanned by $e_1, e_2$ and $Y$ by $f_1, f_2$.

The vector space $\W = V \otimes W$ comes equipped with the symplectic form $(\,,\,) \otimes \langle \,,\,\rangle$.  As in subsection \ref{subsec:Weil}, we set $\X = V \otimes X$, $\Y = V \otimes Y$, $\X_U = U \otimes W \oplus V_0 \otimes X$, $\Y_U = U^\vee \otimes W \oplus V_0 \otimes Y$, so that $\W = \X \oplus \Y = \X_U \oplus \Y_U$ are isotropic decompositions.

We let our groups act on the right of the spaces that define them, i.e., we let $\SO(V)$ act on the right of $V$ and $\Sp(W)$ act on the right of $W$.  Denote by $P_Y = M_Y N_Y$ the Siegel parabolic subgroup of $\Sp(W)$ that is the stabilizer of $Y$ and $M_Y$ the Levi subgroup fixing the decomposition $X \oplus Y$.  Thus $M_Y \simeq \GL(X) \simeq \GL_2$ and we identify $N_Y$ with the symmetric $2 \times 2$ matrices.  Denote by $P_U = M_U N_U$ the Heisenberg parabolic subgroup of $\SO(V)$ that stabilizes the isotropic subspace $U^\vee$ and $M_U$ the Levi subgroup that fixes the decomposition $V = U \oplus V_0 \oplus U^\vee$. Thus $M_U \simeq \GL(U) \times \SO(V_0) \simeq \GL_2 \times \SO(V_0)$. Finally, we denote $Z = [N_U,N_U]$ the center of $N_U$.  Thus, the unipotent group $Z$ is one-dimensional and spanned by the root space for the highest root of $\SO(V)$.

\subsection{The Fourier coefficients of the lift} 
Given a character $\chi: N_U(\Q)\backslash N_U(\A) \rightarrow \C^\times$ of the Heisenberg parabolic of $\SO(V)$, we can compute the $\chi$-Fourier coefficient of a $\theta$-lift:
\begin{align*} 
\theta(f,\phi)_\chi(g) &= \int_{[N_U]}{\chi^{-1}(n)\theta(f,\phi)(ng)\,dn} \\ &= \int_{[N_U] \times [\Sp(W)]}{\chi^{-1}(n)\theta(ng,h;\phi)f(h)\,dh\,dn} \\ &= \int_{[Sp(W)]}{\theta_{\chi}(g,h;\phi)f(h)\,dh}
\end{align*}
where
\[
\theta_{\chi}(g,h;\phi) = \int_{[N_U]}{\chi^{-1}(n)\theta(ng,h;\phi)\,dn}.
\]
If $\phi \in \X_U(\A)$ we can relate $\theta(f,\phi)_\chi(g)$ to the Fourier coefficients of $f$ along $N_Y$, as follows.

Suppose $v_1, v_2 \in V_0$.  We can associate to the pair $v_1, v_2$ a character of $N_U$ and $N_Y$, as follows.  First, one identifies the abelianized unipotent radical $N_U/Z$ with $V_0 \otimes U^\vee$ via the exponential map on the Lie algebra $\so(V) \simeq \wedge^2 V$.  More precisely, if $\delta \in U^\vee$, $v \in V_0$, and $x \in V$, then
\begin{align*} 
x\exp(v \otimes \delta) &= x + x(v \otimes \delta)+ \frac{1}{2} x(v \otimes \delta)^2 \\ &= x + (v,x) \delta - (\delta,x) v - \frac{1}{2}(\delta,x)(v,v) \delta.
\end{align*}
This defines an element of $N_U$, and we denote by $n(v \otimes \delta)$ the associated element of $N_U/Z$. We record now the following formulas:
\begin{align*} 
b_1 \exp(x_1b_{-1} + x_2 b_{-2}) &= b_1 -x_1 - \frac{1}{2}((x_1,x_1) b_{-1} + (x_2,x_1) b_{-2}) \\ b_2 \exp(x_1 b_{-1} + x_2 b_{-2}) &= b_2 - x_2 - \frac{1}{2}((x_1,x_2) b_{-1} + (x_2,x_2) b_{-2}) \\ y \exp(x_1 b_{-1} + x_2 b_{-2}) &= y + (x_1,y)b_{-1} + (x_2,y)b_{-2}.
\end{align*}

Associated to $v_1, v_2$ the element $v_1 \otimes b_1 + v_2 \otimes b_2 \in V_0 \otimes U$ gives a linear map $L_{v_1,v_2}: N_U/Z \simeq V_0 \otimes U^\vee \rightarrow \Q$ defined as 
\[
L_{v_1,v_2}(n(x_{1}\otimes b_{-1} + x_{2} \otimes b_{-2})) = (v_1,x_1) + (v_2,x_2).
\]
We then obtain a character $\chi_{v_1,v_2}: N_U(\Q)\backslash N_U(\A) \rightarrow \C^\times$ as $\chi_{v_1,v_2}(n(x \otimes \delta)) = \psi(L_{v_1,v_2}(x \otimes \delta.))$.

Associated to $v_1, v_2$ we can also define a character on $N_Y$, as follows.  First, using the basis $e_1, e_2, f_1, f_2$, one has an identification $N_Y \simeq Sym^2(X) \simeq Sym^2(\Q^2)$.  If $x \in Sym^2(\Q^2)$ is a $2 \times 2$ symmetric matrix, let $n_Y(x) = \mm{1}{x}{0}{1} \in \Sp(4)$ be the associated unipotent element.  Suppose $T$ is a $2 \times 2$ rational symmetric matrix.  Associated to $T$, we have a character $N_Y(\Q)\backslash N_Y(\A) \rightarrow \C^\times$ as $\chi_T(n_Y(x)) = \psi(\tr(T,x))$.  For an automorphic function $f$ on $\Sp(W)(\A)$, we write
\[
f_T(g) = f_{\chi_T}(g) = \int_{[N_Y]}{\chi_T^{-1}(n)f(ng)\,dn}
\]
the $\chi_T$-Fourier coefficient of $f$.  Finally, if $v_1, v_2 \in V_0$, set 
\[
S(v_1,v_2) = \frac{1}{2}\mb{(v_1,v_1)}{(v_1,v_2)}{(v_1,v_2)}{(v_2,v_2)}
\]
a $2 \times 2$ symmetric matrix.  Thus $\chi_{S(v_1,v_2)}$ is a character of $N_Y(\Q)\backslash N_Y(\A)$.

The next proposition is familiar from work of Piatetski-Shapiro \cite{PSsk} and Rallis \cite{rallisHD}.  For $\phi$ in $S(\X_U(\A))$, it will relate the $\chi_{v_1,v_2}$ Fourier coefficient of $\theta(f,\phi)$ to the $\chi_{S(v_1,v_2)}$ Fourier coefficient of $f$.
\begin{proposition}\label{prop:FC1} Suppose $f$ is a cuspidal automorphic function on $\Sp(W)(\A)$, $\phi \in S(\X_U(\A))$, and $\theta(f,\phi)$ the automorphic function on $\SO(V)$ that is the theta-lift of $f$.  Suppose $v_1, v_2 \in V_0$ and that the pair is non-degenerate in the sense that $S(v_1,v_2)$ has nonzero determinant.  Set
\[
z_{v_1,v_2} = b_1 \otimes f_1 + b_2 \otimes f_2 + v_1 \otimes e_1 + v_2 \otimes e_2 \in \X_U = U \otimes W \oplus V_0 \otimes X.
\]
Then
\begin{equation}\label{eqn:Tchi}
\theta(f,\phi)_{\chi_{-v_1,-v_2}}(g) = \int_{N_Y(\A)\backslash \Sp(W)(\A)}{ \omega_{\psi,\Y_U}(g,h)\phi(z_{v_1,v_2}) f_{-S(v_1,v_2)}(h)\,dh}.
\end{equation}
\end{proposition}

\begin{proof} Following the arguments of \cite[section 5]{PSsk}, first suppose that $g \in \SO(V_0)(\A)$ and $\phi = \phi_U \otimes \phi_{V_0}$, with $\phi_U \in S(U \otimes W)$ and $\phi_{V_0} \in S(V_0 \otimes X)$.  Then
\begin{align*}
\theta(g,h;\phi) &= \sum_{w_1,w_2 \in W}\sum_{y_1, y_2 \in V_0}\omega_{\psi, U^\vee \otimes W}(1,h)\phi_U(b_1 w_1 + b_2 w_2) \omega_{\psi, V_0\otimes Y}(g,h)\phi_{V_0}(y_1 e_1 + y_2 e_2) \\ &= \theta_{V_0 \otimes Y}(g,h;\phi_{V_0}) \left(\sum_{w_1, w_2 \in W}{\omega_{\psi, U^\vee \otimes W}(1,h)\phi_U(b_1 w_1 + b_2 w_2)}\right)
\end{align*}
in obvious notation.

Taking the constant term along $Z \subseteq N_U$, one obtains that
\begin{align} 
\nonumber \theta_Z(g,h) &= \int_{[Z]}{\theta(zg,h)\,dz} \\ \label{eqn:Wsum} &= \theta_{V_0 \otimes Y}(g,h;\phi_{V_0}) \left(\sum_{w_1, w_2 \in W, \langle w_1, w_2 \rangle = 0}{\omega_{\psi, U^\vee \otimes W}(1,h)\phi_U(b_1 w_1 + b_2 w_2)}\right).
\end{align}
This follows from the fact that
\[
\omega_{\psi}(\exp(z b_{-1} \wedge b_{-2}))\phi(b_1 w_1 + b_2 w_2 +y_1 e_1 + y_2 e_2) = \psi(-z \langle w_1, w_2 \rangle) \phi(b_1 w_1 + b_2 w_2 +y_1 e_1 + y_2 e_2),
\]
as is immediately computed from \eqref{eqn:Paction}.

Now, the inner sum in \eqref{eqn:Wsum} can be written as 
\[
\theta^0_U(1,h;\phi_U) + \theta^1_U(1,h;\phi_U) + \theta^2_U(1,h;\phi_U),
\]
where $\theta^j_U(1,h;\phi_U)$ consists of the sum of the terms in \eqref{eqn:Wsum} for which $\dim_{\Q}{Span(w_1,w_2)} = j$.  Moreover, we have
\[
\theta^{2}_U(1,h;\phi_U) = \sum_{\gamma \in N_Y(\Q)\backslash \Sp(W)(\Q)}{\omega(1,\gamma h)\phi_U(b_1f_1 + b_2 f_2)}.
\]
Thus we have
\[
\theta_Z(g,h) = \theta_{V_0 \otimes Y}(g,h)(\theta^0_U(1,h;\phi_U) + \theta^1_U(1,h;\phi_U)) + \theta_{V_0 \otimes Y}(g,h) \theta^2_U(1,h;\phi_U).
\]

Consider an element $x_1 b_{-1} + x_2 b_{-2} \in V_0 \otimes U^\vee$.  Set
\[
z:=(b_1 w_1 + b_2 w_2 + y_1 e_1 + y_2 e_2) \exp(x_1 b_{-1} + x_2 b_{-2})
\]
Then
\begin{lemma}\label{lem:w1w2} Suppose $w_1, w_2 \in Y$.  Then
\[
\frac{1}{2} \langle pr_{\X_U}(z), pr_{\Y_U}(z) \rangle = (x_1, y_1) \langle w_1, e_1 \rangle + (x_1,y_2) \langle w_1, e_2 \rangle + (x_2, y_1) \langle w_2, e_1 \rangle + (x_2, y_2) \langle w_2, e_2 \rangle.
\]
\end{lemma}
\begin{proof} We have 
\begin{align*}
z &= (b_1 -x_1 - \frac{1}{2}((x_1,x_1) b_{-1} + (x_2,x_1) b_{-2})) w_1  + (b_2 - x_2 - \frac{1}{2}((x_1,x_2) b_{-1} + (x_2,x_2) b_{-2})) w_2 \\ &\;\; + (y_1 + (x_1,y_1)b_{-1} + (x_2,y_1)b_{-2}) e_1 + (y_2 + (x_1,y_2)b_{-1} + (x_2,y_2)b_{-2})e_2.
\end{align*}
Thus if $w_1, w_2 \in Y$,
\[
pr_{\X_U}(z) = b_1 w_1 + b_2 w_2 + y_1 e_1 + y_2 e_2
\]
and
\begin{align*} 
pr_{\Y_U}(z) &= -(x_1+ \frac{1}{2}((x_1,x_1) b_{-1} + (x_2,x_1) b_{-2})) w_1 -(x_2 + \frac{1}{2}((x_1,x_2) b_{-1} + (x_2,x_2) b_{-2})) w_2 \\ &\;\; + ((x_1,y_1)b_{-1} + (x_2,y_1)b_{-2}) e_1 + ((x_1,y_2)b_{-1} + (x_2,y_2)b_{-2})e_2.
\end{align*}
The lemma follows easily.
\end{proof}

From the fact that the pair $v_1, v_2$ is non-degenerate, applying Lemma \ref{lem:w1w2} one sees that the terms with $\theta_U^0$ and $\theta_U^1$ vanish upon taking the $\chi_{v_1,v_2}$-Fourier coefficient.  Thus we obtain
\[
\theta(f;\phi)_{\chi_{-v_1,-v_2}}(g) = \int_{N_Y(\Q)\backslash \Sp(W)(\A)}{\omega_{\psi}(g,h)\phi(z_{v_1,v_2}) f(h)\,dh}.
\]
Let $s = \mm{s_1}{s_2}{s_2}{s_3}$.  Then 
\[
 z_{v_1,v_2} n_Y(s) = b_1 f_1 + b_2 f_2 + v_1(e_1 + s_1 f_1 + s_2 f_2) + v_2(e_2 + s_2 f_1 + s_3 f_2)
\]
and thus
\[
\frac{1}{2}\langle pr_{\X_U}(z_{v_1,v_2}n_Y(s) ), pr_{\Y_U}( z_{v_1,v_2}n_Y(s))\rangle = \tr(S(v_1,v_2) s).
\]
The statement of the proposition now follows in the restricted setting $\phi = \phi_U \otimes \phi_{V_0}$ and $g \in \SO(V_0)$.  

For the general case, one reduces to the above special case.  Indeed, we have proved \eqref{eqn:Tchi} when $g=1$ (or more generally, $g \in \SO(V_0)$) and $\phi = \phi_U \otimes \phi_{V_0}$.  By linearity, the proposition follows for $g=1$ and all $\phi$. Finally, defining $\phi' = \omega_{\psi,\Y_U}(g,1)\phi$, \eqref{eqn:Tchi} for $\phi'$ and $g=1$ gives \eqref{eqn:Tchi} for $\phi$ and $g$.  This completes the proof of the proposition.
\end{proof}

Let us now consider the integral \eqref{eqn:Tchi} when $\phi = F(\phi')$ for $\phi' \in S(\X(\A))$.  More precisely, in section \ref{sec:Arch}, we will need an expression for $\omega_{\psi, \Y_U}(1,h)F(\phi')(z_{v_1,v_2})$ when $h \in M_Y(\A)$.  We compute this now. For $h \in M_Y$, we abuse notation and let $\det(h)$ denote the determinant of $h$ acting on $X$. 
\begin{lemma}\label{lem:PFm} Suppose $\phi' \in S(\X(\A))$, $\phi = F(\phi')$ is in $S(\X_{U}(\A))$ and $h \in M_Y(\A)$.  Then
\begin{align*}
\omega_{\psi,\Y_U}(1,h)&\phi(z_{v_1,v_2}) =\\ & |\det(h)|^{(\dim(V_0)+4)/2}\int_{X^2}\psi(\langle z_1, f_1 \rangle + \langle z_2, f_2 \rangle)\phi'((v_1e_1 + v_2e_2 + b_{-1}z_1 + b_{-2}z_2)h)\,dz_1\,dz_2.
\end{align*}
\end{lemma}
\begin{proof} We have 
\[
\omega_{\psi,\Y_U}(1,h)\phi(z_{v_1,v_2}) = |\det(h)|^{\dim V_0/2}\phi( b_1(f_1 h) + b_2(f_2 h) + v_1 (e_1h) + v_2(e_2 h)).
\]
Thus
\begin{align*}
\omega_{\psi,\Y_U}(1,h)\phi(z_{v_1,v_2}) &= |\det(h)|^{\dim V_0/2}\int_{U^\vee \otimes X}\psi(\langle z', b_1(f_1h) + b_2(f_2h) \rangle)\phi'(v_1(e_1h) + v_2(e_2h) + z')\,dz' \\ &= |\det(h)|^{\dim V_0/2}\int_{X^2}\psi(\langle b_{-1}z_1 + b_{-2}z_2, b_1(f_1h) + b_2(f_2h) \rangle) \\ &\;\;\;\;\;\;\;\;\; \times \phi'(v_1(e_1h) + v_2(e_2h) + b_{-1}z_1+b_{-2}z_2)\,dz_1\,dz_2 \\ &= |\det(h)|^{(\dim(V_0)+4)/2}\int_{X^2}\psi(\langle z_1, f_1 \rangle + \langle z_2, f_2 \rangle) \\ &\;\;\;\; \times \phi'((v_1e_1 + v_2e_2 + b_{-1}z_1 + b_{-2}z_2)h)\,dz_1\,dz_2.
\end{align*}
In the last line we have made the variable change $z_j \mapsto z_j h$ for $j = 1,2$. This gives the lemma.\end{proof}

\subsection{Theta lift of Poincare series}\label{subsec:PS1}
Finally, we discuss the theta lift of certain Poincare series, in an abstract, formal setting.  Suppose we have a non-degenerate symmetric $2 \times 2$ rational matrix $T$, and $\chi_T$ denotes the associated character of $N_Y(\Q)\backslash N_Y(\A)$.  Suppose that $\mu_T: \Sp(W)(\A) \rightarrow \C$ is a function satisfying $\mu_T(n_Y(s)h) = \psi((T,s))\mu_T(h)$ for all $h \in \Sp(W)(\A)$.  The Poincare series associated to $\mu_T$ is the automorphic function
\[P(h;\mu_T) = \sum_{\gamma \in N_Y(\Q)\backslash \Sp(W)(\Q)}{\mu_T(\gamma h)},\]
if the sum converges absolutely. 

In section \ref{sec:Arch}, we will prove that for a particular special choice of archimedean test data $\phi_\infty \in S(\X(\R))$, the theta lift of holomorphic Siegel modular forms is a quaternionic modular form on $\SO(V)$.  The proof follows the method of Oda \cite{oda} and Niwa \cite{niwa}, whereby one proves that the lifts of certain Poincare series on $\Sp(W)$ are quaternionic on $\SO(V)$, and deduces the general case from the fact that the Poincare series span the cuspidal Siegel modular forms \cite[Chapter 3]{klingen}.  We now write out the formal calculation.

\begin{lemma} Suppose the sum defining $P(h,\mu_T)$ converges absolutely to a cuspidal automorphic form on $\Sp(W)(\A)$. Let $\phi \in S(\X(\A))$ and suppose moreover that
\begin{equation} \label{eqn:finite}
\sum_{y_1, y_2 \in V}\int_{N_Y(\Q)\backslash \Sp(W)(\A)}{|\mu_T(h)||(\omega_{\psi}(g,h)\phi)(y_1 e_1 + y_2 e_2)|\,dh}
\end{equation}
is finite.  Then 
\[
\theta(P(\cdot;\mu_T);\phi)(g) = \sum_{y_1, y_2 \in V, S(y_1,y_2) = -T} \left(\int_{N_Y(\A)\backslash \Sp(W)(\A)}{\mu_T(h)(\omega_{\psi,\Y}(g,h)\phi)(y_1 e_1 + y_2 e_2)\,dh}\right).
\]
\end{lemma}
\begin{proof} Because $P(h;\mu_T)$ is assumed cuspidal, the integral defining the theta lift converges absolutely.  Computing formally,
\begin{align*}
\theta(P(\cdot;\mu_T);\phi)(g) &= \int_{[Sp(W)]}{\theta(g,h;\phi)P(h;\mu_T)\,dh} \\ &= \int_{N_Y(\Q)\backslash \Sp(W)(\A)}{\theta(g,h;\phi)\mu_T(h)\,dh} \\ &= \int_{N_Y(\A)\backslash \Sp(W)(\A)}{\mu_T(h)\left(\int_{[N_Y]}{\psi((T,s))\theta(g,n_Y(s)h;\phi)\,ds}\right)\,dh}.
\end{align*}
The finiteness assumption of the lemma proves that the above formal manipulations are justified.  Now
\[
\theta(g,h;\phi) = \sum_{y_1, y_2 \in V}{\omega_{\psi,\Y}(g,h)\phi(y_1 e_1 + y_2 e_2)}
\]
and $\omega_{\psi,\Y}(1,n_Y(s))\phi_0(y_1 e_1 + y_2 e_2) = \psi(S(y_1,y_2),s)\phi_0(y_1e_1 + y_2 e_2)$ for any $\phi_0 \in S(\X)$.  One obtains
\[
\int_{[N_Y]}{\psi((T,s))\theta(g,n_Y(s)h;\phi)\,ds} = \sum_{y_1, y_2 \in V, S(y_1,y_2) = -T}{\omega_{\psi,\Y}(g,h)\phi(y_1 e_1 + y_2 e_2)}.
\]
The lemma follows.
\end{proof}

\section{Special archimedean test data}\label{sec:Arch} This section is almost entirely archimedean.  We define the special test data in $S(\X(\R))$ and prove that the theta lift of weight $N:=w+2-n/2$ Siegel modular forms on $\Sp(W)$ are weight $w$ modular forms on $\SO(V)$, in the sense of \cite{pollackQDS}.  We also prove a certain archimedean result that is crucial to deducing the algebraicity of the Fourier coefficients of the lift $\theta(f;\phi)$ on $\SO(V)$.  Finally, we define the Poincare series $P(h;\mu_T)$ that we will use, and prove a bound that is required to justify the finiteness of \eqref{eqn:finite}.

Denote by $V_2 = \C^2$ the defining representation of $\SU(2)$.  Throughout the section, $K_V$ denotes the subgroup of $\SO(V)(\R)$ that also fixes the majorant $||\cdot ||^2$.  This is a maximal compact subgroup.  Combining the definitions and normalizations of \cite{pollackQDS} Appendix A and Section 5.1, we obtain a map 
\[p_K: \wedge^2 V \otimes\C \rightarrow \mathfrak{su}_2 \otimes \C \simeq \sl_2(\C) \simeq Sym^2(V_2).\]
Our fixed long-root $\SU(2) \subseteq K_V$ is defined to be the subgroup of $K_V$ with Lie algebra the image $p_K(\wedge^2 V \otimes \R)$.  The map $p_K$ also induces a map $K_V \rightarrow \SU(2)/\mu_2$.  Set $\Vm_w = Sym^{2w}V_2$, which we think of as a representation of $K_V$ via this map.  \emph{Modular forms} on $\SO(V)$ and $G_2$ of weight $w$ are certain $\Vm_w$-valued functions on $\SO(V)(\A)$ and $G_2(\A)$. 

\subsection{Test data} In this subsection we define a certain $\Vm_w$-valued Schwartz function $\phi_\infty$ on $\X(\R)$, and prove various properties of it that are crucial to what follows.  Specially, for $y_1, y_2 \in V$ we set
\[\phi_\infty(y_1,y_2) := \phi_\infty(y_1e_1 + y_2 e_2) = p_K(y_1 \wedge y_2)^w e^{-2\pi (||y_1||^2 + ||y_2||^2)}.\]
Note that $p_K(y_1 \wedge y_2) \in \Vm_1$, and we consider $p_K(y_1 \wedge y_2)^w$ as an element of $\Vm_w$ via the $K_V$-equivariant map $Sym^w \Vm_1 \rightarrow \Vm_w$. Note also that there is a $2\pi$ in the exponential factor--as opposed to just a $\pi$--because we have a factor of $\frac{1}{2}$ in our definition of $||y||^2$.  Immediately from the definition, one has 
\[
\phi_\infty((y_1e_1 + y_2 e_2)gk) = k^{-1} \phi_\infty((y_1 e_1 + y_2 e_2)g)
\]
for $g \in \SO(V)(\R)$ and $k \in K_V$.

Denote by $V_{+}$ is the four-dimensional subspace of $V$ where $\iota$ acts by $1$ and $V_{-}$ is the $n$-dimensional subspace of $V$ where $\iota$ acts by $-1$.  For $y \in V = V_+ \oplus V_{-}$, write $y = y_+ + y_{-}$, where $y_+ \in V_{+}$ and $y_- \in V_{-}$.  Then we have
\begin{equation}\label{eqn:pm}
\phi_\infty(y_1,y_2) = p_K(y_{1,+}\wedge y_{2,+})^w e^{-2\pi(||y_{1,+}||^2 + ||y_{2,+}||^2)} e^{-2\pi(||y_{1,-}||^2 + ||y_{2,-}||^2)}.
\end{equation}
  
Recall the notion of a \emph{pluriharmonic} function from \cite[page 4]{kV}.
\begin{lemma}\label{lem:pluri} The $\Vm_w$-valued polynomial $x_1 e_1 + x_2 e_2 \mapsto p_K(x_1 \wedge x_2)^{w}$ on $V_{+} \otimes X$ is pluriharmonic.\end{lemma}
We remark that if $m \in \GL_2(\C)$, one has 
\begin{equation}\label{eqn:pmwdet}
x_1(e_1m) + x_2(e_2m) \mapsto \det(m)^w p_K(x_1 \wedge x_2)^{w}.
\end{equation}
\begin{proof} Note that the polynomial $x_1 e_1 + x_2 e_2 \mapsto p_K(x_1 \wedge x_2)^{w}$ is of degree $2w$.  As $2w$ is the smallest integer $k$ for which $\Vm_w$ can occur in $Sym^k(X \otimes V_{+})$, the pluriharmonicity follows from \cite[Corollary 5.4]{kV}. \end{proof}

Let $K_W \subseteq \Sp(W)(\R)$ be the maximal compact subgroup that fixes the inner product on $W$ for which $e_1, e_2, f_1, f_2$ is an orthonormal basis.  Then $K_W \simeq U(2)$ via the map $\mm{A}{B}{-B}{A} \mapsto A+iB$.  Denote by $\H_2$ the Siegel upper half-space of degree two, and write $j: \Sp(W)(\R) \times \H_2 \rightarrow \C$ for the usual factor of automorphy $j(g,Z) = \det(cZ+d)$ for $g = \mm{a}{b}{c}{d}$.  Then if $k \in K_W \simeq U(2)$, the representation $\det(\cdot)^N$ of $U(2)$ is realized as $k \mapsto j(k,i)^{-N}$.

\begin{corollary}\label{cor:Kequiv} For $k \in K_W$, one has $\omega_{\psi,\Y}(1,k)\phi_\infty(y_1,y_2) = j(k,i)^{-(w+2-n/2)} \phi_\infty(y_1,y_2)$.\end{corollary}
\begin{proof} This follows from \eqref{eqn:pm}, \eqref{eqn:pmwdet} and the pluriharmoncity of Lemma \ref{lem:pluri}.  See, e.g., \cite[section 2.5]{lionVergne}. \end{proof}

\subsection{Poincare series}\label{subsec:Poincare} We now define the Poincare series.  Classically, the holomorphic Poincare series of weight $N$ (of exponential type) are defined as 
\[
P_N(Z;T) = \sum_{\gamma \in \Gamma_\infty\backslash \Gamma}{j(\gamma,Z)^{-N} e^{2\pi i (T,\gamma Z)}}.
\]
Here $T$ is a fixed positive-definite half-integral symmetric matrix, $\Gamma = \Sp_4(\Z)$, $\Gamma_\infty = \{\gamma = \mm{1}{V}{0}{1} \in \Gamma\}$ and $N >>0$.  The sum converges absolutely to a \emph{cuspidal} holomorphic Siegel modular form of weight $N$.  See \cite[Chapter 3]{klingen}.  

Adellically, one proceeds as follows.  Fix $T$ a positive definite two-by-two real symmetric matrix and an integer $N >>0$. Define $\mu_{T,\infty}^+$ on $\Sp(W)(\R)$ as $\mu_{T,\infty}^+(h) = j(h,i)^{-N} e^{2\pi i (T, h \cdot i)}$.  For $p$ a finite prime, define $\mu_{T,p}^+$ on $\Sp(W)(\Q_p)$ as any function supported on $N_Y(\Q_p)\Sp(W)(\Z_p)$ satisfying $\mu_{T,p}^+(n_Y(s)h) = \psi_p((T,s))\mu_{T,p}^+(h)$ for all $h \in \Sp(W)(\Q_p)$ and all $n_Y(s) \in N_Y(\Q_p)$.  Assume moreover that at all but finitely many primes, $\mu_{T,p}^+$ is the function $\mu_{T,p}^+(n_Y(s)k) = \psi_p((T,s))$.  With these assumptions, we define $\mu_{T}^+$ on $\Sp(W)(\A)$ as $\mu_{T}^+(h) = \prod_{v}{\mu_{T,v}^+(h_v)}$.  Then $\mu_{T}^+(n_Y(s) h) = \psi((T,s))\mu_T^+(h)$ for all $h \in \Sp(W)(\A)$ and $s \in N_Y(\A)$, and
\[
P(h;\mu_{T}^+) = \sum_{\gamma \in N_Y(\Q)\backslash \Sp(W)(\Q)}{\mu_T^+(\gamma h)}
\]
is a Poincare series that corresponds to a cuspidal holomorpic Siegel modular form on $\Sp(W)$ of weight $N$.  We now set $\mu_{-T}(h) = \overline{\mu_{T}^{+}(h)}$ and $P(h;\mu_{-T}) = \sum_{\gamma \in N_Y(\Q)\backslash \Sp(W)(\Q)}{\mu_{-T}(\gamma h)}$, as in subsection \ref{subsec:PS1}.  Then $\mu_{-T}(n_Y(s)h) = \psi(-(T,s))\mu_{-T}(h)$ and $P(h;\mu_{-T}) = \overline{P(h;\mu_{T}^{+})}$ is the complex conjugate of a holomorphic Siegel modular form on weight $N$.  We will compute the $\theta$-lift of $P(h;\mu_{-T})$.

To do this, we compute the following integral.  For $X \in \wedge^2 V \simeq \so(V)$ such that $p_K(X) \neq 0$, define
\[
A_w(X) = \frac{p_K(X)^w}{||p_K(X)||^{2w+1}}.
\]
Here $|| \cdot ||$ is the $K_V$-invariant norm on $\so(V)$ induced from the Cartan involution $\theta_\iota$, where $\theta_\iota(y_1 \wedge y_2) = \iota(y_1) \wedge \iota(y_2)$.  The function $A_w$ will play an important role in what follows.  See also \cite[Section 6]{pollackG2}, \cite[section 2.1]{pollackE8} where the function $A_w$ appears in the construction of degenerate Heisenberg Eisenstein series.

\begin{proposition}\label{prop:Awint} Suppose $N= w+2-n/2$ is the weight of the Poincare series and $y_1, y_2 \in V$ with $S(y_1,y_2) = T > 0$.  Then $p_K(y_1 \wedge y_2) \neq 0$ and
\[
I_\infty(y_1,y_2;T):=\int_{N_Y(\R)\backslash \Sp(W)(\R)}{\mu_{-T}(h)\omega_{\psi,\Y}(1,h)\phi_\infty(y_1 e_1 + y_2 e_2)\,dh}\]
is equal to $A_w(y_1 \wedge y_2)$ up to a nonzero constant which is independent of $y_1, y_2$.\end{proposition}
\begin{proof} By construction, $\mu_{-T}(hk) = j(k,i)^N \mu_{-T}(h)$ for all $h \in Sp(W)(\R)$ and $k \in K_W \simeq U(2)$.  By Corollary \ref{cor:Kequiv}, $\omega_{\psi,\Y}(1,hk)\phi_\infty(y_1,y_2) = j(k,i)^{-N}\omega_{\psi,\Y}(1,h)\phi_\infty(y_1,y_2)$.   Thus applying the Iwasawa decomposition, we obtain
\begin{align*}
I_\infty(y_1,y_2;T) &= \int_{M_Y(\R)}{\delta_P^{-1}(h)\mu_{-T}(h)\omega_{\psi,\Y}(1,h)\phi_\infty(y_1, y_2)\,dh} \\ &= p_K(y_1 \wedge y_2)^{w} \int_{\GL_2(\R)}{|\det(m)|^{-3+w+N+\dim(V)/2}e^{-2\pi (T + R(y_1,y_2), m \,^tm)}\,dm}.
\end{align*}
Here
\[
R(y_1,y_2) = \frac{1}{2}\mb{(y_1,y_1)_r}{(y_1,y_2)_r}{(y_1,y_2)_r}{(y_2,y_2)_r}
\]
and $(y,y')_r := (y,\iota(y'))$ is the positive definite majorant on $V$.  We have also used that $4|n$ so that $w+N$ is even and thus $\det(m)^{w+N} = |\det(m)|^{w+N}$.

Note that $w+N+\dim(V)/2-3 = 2w+1$.  Thus
\[
I_\infty(y_1,y_2;T) = \frac{p_K(y_1\wedge y_2)^{w}}{|\det(T + R(y_1,y_2))|^{w+1/2}} \left(\int_{\GL_2(\R)}{|\det(m)|^{2w+1} e^{-2\pi\tr(m\,^tm)}\,dm}\right).
\]
We remark that this latter integral is a so-called Siegel integral.  The proposition thus follows from the following lemma.
\end{proof}

\begin{lemma} Suppose $y_1, y_2 \in V$ and $S(y_1, y_2) > 0$.  Then $2 ||p_K(y_1 \wedge y_2)||^2 =  \det(S(y_1,y_2) + R(y_1,y_2))$, and in particular, $p_K(y_1 \wedge y_2) \neq 0$. \end{lemma}
\begin{proof} One has $\frac{1}{2}(y,y') + \frac{1}{2}(y,\iota(y')) = (y_{+},y'_{+})$.  Thus
\[\det(S(y_1,y_2)+R(y_1,y_2)) = \det\left(\mb{(y_{1,+},y_{1,+})}{(y_{1,+},y_{2,+})}{(y_{1,+},y_{2,+})}{(y_{2,+},y_{2,+})}\right).\]
This determinant is 
\[
||y_{1,+} \wedge y_{2,+}||^2 = 2||p_K(y_{1,+} \wedge y_{2,+})||^2 = 2||p_K(y_1 \wedge y_2)||^2
\]
giving the lemma.
\end{proof}

\subsection{The function $A_w$} In this subsection, we prove that the function $g \mapsto A_{w}(y_1 \wedge y_2 g)$ is quaternionic, i.e., that it is annihilated by $D_w$.  This is the key step in showing that the theta lifts $\theta(f;\phi)$ are quaternionic modular forms.

Recall the differential operator $D_w$ that defines modular forms of weight $w$. For $Z = y_1 \wedge y_2$ with $S(y_1,y_2) > 0$, define $B_{w,Z}(g): \SO(V)(\R) \rightarrow \Vm_w$ as $B_{w,Z}(g) = A_{w}(Zg)$. It is clear that $B_{w,Z}(gk) = k^{-1} B_{w,Z}(g)$ for all $g \in \SO(V)(\R)$ and $k \in K_V$.  The following theorem is crucial to all that follows.

\begin{theorem}\label{thm:quat} Suppose $w \geq 2$.  With notation as above, the function $B_{w,Z}(g)$ is quaternionic, i.e., $D_w B_{w,Z}(g) \equiv 0$.\end{theorem}
\begin{proof} We begin with a simple lemma.  For $L \in \mathfrak{su}_2 \otimes \C \simeq \sl_2(\C) \simeq Sym^2(V_2)$, denote by $(L,L) \in \C$ the $\SL_2(\C)$-invariant quadratic form on $L$.  Thus, if $L \in \mathfrak{su}_2$, then $(L,L) = ||L||^2$.
\begin{lemma} Suppose $X \in \so(V)$, and $X \cdot B_{w,Z}(g) = \frac{d}{dt}\left(B_{w,Z}(ge^{tX})\right)|_{t=0}$ denotes the right regular action. Set $Z' = Zg$. Then
\begin{align*}
X\cdot & B_{w,Z}(g) = \\ & \frac{p_K(Z')^{w-1}}{||p_K(Z)'||^{2w+3}}\left(w p_K([Z',X])(p_K(Z'),p_K(Z')) -(2w+1)p_K(Z')(p_K([Z',X]),p_K(Z'))\right).
\end{align*}
\end{lemma}
\begin{proof} This follows immediately from the definitions.  Note that the quantity outside the parentheses is an element of $\Vm_{w-1}$, the quantity inside the parenthesis is an element of $\Vm$, and the product is considered as an element of $\Vm_w$.\end{proof}

To continue computing, we now choose an isomorphism 
\begin{equation}\label{eqn:VC}
V \otimes \C \simeq V_2^{(1)} \otimes V_2^{(2)} \oplus V' = V_2^{(1)} \otimes e \oplus V' \oplus V_2^{(1)} \otimes f,
\end{equation}
as in \cite[section A.3]{pollackQDS} so that $V_+ \otimes\C$ maps over to $V_2 \otimes V_2$ and $V_{-} \otimes \C$ maps over to $V'$.  More precisely, $V_+ \otimes \C$ is a non-degenerate split quadratic space of dimension four, so it it can be identified with $(M_2 = V_2^{(1)} \otimes V_2^{(2)}, \det)$ the space of $2 \times 2$ matrices $M_2$ with determinant as quadratic form.  Here $V_2^{(1)}$, $V_2^{(2)} = \Span\{e,f\}$ are two copies of the two-dimensional reperesentation of $\SL_2$, and we identify $M_2$ with $V_2^{(1)} \otimes V_2^{(2)}$ as in \cite[section A.3]{pollackQDS}.

Fix $w_1 = w_1^e e + w_1^v + w_1^f f$ and $w_2 = w_2^e e + w_2^v + w_2^f f$ in $V \otimes \C$, written in terms of the decomposition \eqref{eqn:VC}, so that $w_j^e, w_j^f \in V_2^{(1)}$ and $w_j^v \in V'$ for $j=1,2$.  Set $p = w_1^e w_2^f - w_2^e w_1^f = p_K(w_1\wedge w_2)$, considered as an element of $\Vm_1 = Sym^2(V_2).$  If $X \in \wedge^2(V \otimes \C)$, write $X(p) := p_K([w_1 \wedge w_2, X])$ for shorthand.  

Set $W_J = V_{2}^{(2)}\otimes V'$.  Denote by $\{X_\gamma\}_\gamma$ a basis of $\p \otimes \C = V_{2}^{(1)} \otimes W_J$ and $\{X_\gamma^\vee\}_\gamma$ the dual basis of $\p^\vee$.  With this notation,
\[
X_\gamma \cdot B_{w,Z}(g) =  \frac{p^{w-1}}{||p||^{2w+3}}\left(w (p,p) X_\gamma(p) - (2w+1)(X_\gamma(p),p) p\right).
\]
Contracting with $X_\gamma^\vee$, one obtains $\frac{1}{||p||^{2w+3}}$ times
\[
wp^{w-2}\left( (w-1)(p,p) \langle p, X_\gamma^\vee \rangle X_\gamma(p) + (p,p)\langle X_\gamma(p), X_\gamma^\vee \rangle p - (2w+1) (X_\gamma(p),p) \langle p, X_\gamma^\vee \rangle p\right).
\]
Theorem \ref{thm:quat} thus follows from the following proposition. \end{proof}

\begin{proposition} Let the notation be as above.  Then
\begin{equation}\label{eqn:sum0}
\sum_{\gamma}{(w-1)(p,p) \langle p, X_\gamma^\vee \rangle X_\gamma(p) + (p,p)\langle X_\gamma(p), X_\gamma^\vee \rangle p - (2w+1) (X_\gamma(p),p) \langle p, X_\gamma^\vee \rangle p}
\end{equation}
is $0$ as an element of $Sym^3(V_2) \otimes W$.
\end{proposition}
\begin{proof} Let $\{v_\alpha\}_\alpha$ be a basis of $V_2$ and $\{w_\beta\}_\beta$ a basis of $W_J$ so that $X_\gamma = v_\alpha \otimes w_\beta$ is a basis of $\p \otimes \C = V_2 \otimes W_J$.  Of course, $V_2$ is two-dimensional, so the sum over $\alpha$ has two terms.  

To check the vanishing in \eqref{eqn:sum0}, we can fix $\beta$ and sum over $\alpha$.  Then the required vanishing follows from the following three lemmas.

For $w_1, w_2 \in V$, set 
\begin{align*}
(w_1 \wedge w_2)_\p &= w_1^{-} \wedge w_2^+ + w_1^{+} \wedge w_2^{-} \\ &= w_1^v \wedge (w_2^e e + w_2^f f) + (w_1^e e + w_1^f f) \wedge w_2^v \\ &= w_1^e \otimes (ew_2^v) + w_1^f \otimes (f w_2^v) - w_2^e \otimes (e w_1^v) - w_2^f \otimes (f w_1^v)
\end{align*}
This is an element of $\p = V_2^{(1)} \otimes (V_2^{(2)} \otimes V')$.  

The $\langle \,,\, \rangle_W$ below denote $W_J$-contractions.  Thus $\langle w_\beta, (w_1 \wedge w_2)_\p \rangle_W$ denotes a $W_J$-contraction between $w_\beta$, an element of $W_J$, and $(w_1 \wedge w_2)_\p$, an element of $V_{2}^{(1)} \otimes W_J$, so that $\langle w_\beta, (w_1 \wedge w_2)_\p \rangle_W \in V_2^{(1)}$.

Note that if $X_\gamma = v_\alpha w_\beta$, then
\begin{align*}
X_\gamma(p) &= p_K([w_1 \wedge w_2, v_\alpha \wedge w_\beta]) \\ &= p_K([(w_1^e e + w_1^f f) \wedge w_2^v + w_1^v \wedge (w_2^e e + w_2^f f), v_\alpha w_\beta]) \\ &= \langle w_\beta, ew_2^v \rangle_W w_1^e v_\alpha - \langle w_\beta, e w_1^v \rangle_W w_2^e v_\alpha + \langle w_\beta, f w_2^v \rangle_W w_1^f v_\alpha - \langle w_\beta, f w_1^v \rangle_W w_2^f v_\alpha.
\end{align*}
Thus $X_\gamma(p) = \langle w_\beta, (w_1\wedge w_2)_{\p}\rangle_W v_\alpha$.  For ease of notation, set $v_p^\beta = \langle w_\beta, (w_1 \wedge w_2)_\p\rangle_W$, which is an element of $V_2 = V_2^{(1)}$.

\begin{lemma} Let the notation be as above.  Then
\[\sum_{\alpha}{\langle X_\gamma(p), X_\gamma^\vee \rangle} = 3 \langle w_{\beta}, (w_1 \wedge w_2)_{\p} \rangle_W w_\beta^\vee.\]
\end{lemma}
\begin{proof} Let $\alpha_1, \alpha_2$ be our basis of $V_2 = V_2^{(1)}$.  Then
\begin{align*}
\langle X_\gamma(p), X_\gamma^\vee \rangle &= \langle v_{p,\beta} v_\alpha, v_\alpha^\vee w_\beta^\vee \rangle \\ &= (v_{p,\beta} + \langle v_{p,\beta}, v_\alpha^\vee \rangle ) w_\beta.
\end{align*}
Taking $\alpha = \alpha_1, \alpha_2$ and summing up gives $3 v_{p}^\beta w_\beta^\vee$, which is the statement of the lemma.\end{proof}

\begin{lemma}Let the notation be as above.  Then
\[\sum_{\alpha}{\langle p, X_\gamma^\vee \rangle X_\gamma(p)} = 2 \langle w_{\beta}, (w_1 \wedge w_2)_{\p} \rangle_W p w_\beta^\vee.\]
\end{lemma}
\begin{proof} This follows immediately from the fact that $\sum_{j}{\langle p, v_{\alpha_j}^\vee \rangle v_{\alpha_j}} = 2p$.\end{proof}

\begin{lemma}Let the notation be as above. Then
\[\sum_{\alpha}{ (p, X_\gamma(p)) \langle p, X_\gamma^\vee \rangle}  = (p,p) \langle w_\beta, (w_1 \wedge w_2)_{\p}\rangle_W w_\beta^\vee.\]
\end{lemma}
\begin{proof} We have
\[
(p,X_\gamma(p))\langle p, X_\gamma^\vee \rangle = (p, v_{p}^\beta v_\alpha) \langle p, v_\alpha^\vee \rangle w_\beta^\vee.
\]

We thus must check the identity $\sum_{\alpha}{(p,v_p^\beta v_\alpha) \langle p, v_\alpha^\vee \rangle} = (p,p) v_p^\beta$.  To do this, first note that for any $u, v, X \in V_2$, one has 
\begin{equation}\label{eqn:uvX}
\langle u, X\rangle v - \langle v, X \rangle u = \langle u,v \rangle X.
\end{equation}
Now, we claim
\[ \sum_{\alpha}{(p, v_p^\beta v_\alpha) v_\alpha^\vee} = -\langle p, v_p^\beta \rangle.\]
To check this, we may assume $p = \ell_1 \ell_2$ is a product of two linear factors, and then the left-hand side is
\begin{align*}
\sum_{\alpha}{(p,v_p^\beta v_\alpha) v_\alpha^\vee} &= (\ell_1 \ell_2, v_p^\beta e) f - (\ell_1 \ell_2, v_p^\beta f) e \\ &= (\langle \ell_1, v_p^\beta \rangle \langle \ell_2, e \rangle + \langle \ell_1, e \rangle \langle \ell_2, v_p^\beta \rangle) f  - (\langle \ell_1, v_p^\beta \rangle \langle \ell_2, f \rangle + \langle \ell_1, f \rangle \langle \ell_2, v_p^\beta \rangle) e \\ &= \langle \ell_1, v_p^\beta \rangle (-\ell_2) + \langle \ell_2, v_p^\beta \rangle (- \ell_1) \\ &= - \langle p, v_p^\beta \rangle.
\end{align*}

Thus, to conclude, we must check that $\langle p, \langle v_p^\beta, p \rangle \rangle = (p,p) v_p^\beta$.  For this, linearizing, it suffices to check that 
\[\langle \ell_1 \ell_2, \langle v, \ell_1' \ell_2' \rangle \rangle + \langle \ell_1' \ell_2', \langle v, \ell_1 \ell_2 \rangle \rangle = 2 (\ell_1 \ell_2, \ell_1' \ell_2') v
\]
for all $\ell_1, \ell_2, \ell_1', \ell_2', v \in V$.  Expanding the left-hand side, one checks this using the identity \eqref{eqn:uvX} four times.  This proves the proposition, and with it, Theorem \ref{thm:quat}.\end{proof}
\end{proof}

\subsection{The Fourier transform of $A_w$} In this subsection, we analyze a certain Fourier transform integral of the function $B_{w,Z}(g)$.  The result is essential to proving that the $\theta$-lifts have algebraic Fourier coefficients.

For $y_1, y_2 \in V_0$, what we need to compute is
\[\int_{N_{y_1, y_2}(\R)}{A_{w}(y_1 \wedge y_2 n g)\chi_{y_1, y_2}^{-1}(n)\,dn}.\]
Here $N_{y_1,y_2} = (\mathrm{Stab}(y_1,y_2)\cap N_U)\backslash N_U$ is $4$-dimensional and abelian.  As a function of $g$, the above integral is quaternionic, i.e it is annihilated by $D_w$.  Thus by the multiplicity one result of \cite{pollackQDS}, we know that the integral is $C_{y_1,y_2,w} W_{\chi_{y_1,y_2}}(g)$ for some constant $C_{y_1,y_2,w}$ depending on $y_1, y_2, w$.  Thus, we must compute the above integral for $g=1$ so that we can obtain $C_{y_1, y_2, w}$.

To pin down the normalizations, we proceed to compute the following:
\[
J(y_1,y_2;w):=\int_{M_{2}(\R)}{A_{w}((y_1 + r_{11} b_{-1} + r_{12} b_{-2}) \wedge (y_2 + r_{21} b_{-1} + r_{22} b_{-2}))\psi^{-1}(\tr(r))\,dr}.
\]
Here again $y_1, y_2 \in V_0$ and $r = \mm{r_{11}}{r_{12}}{r_{21}}{r_{22}} \in M_2(\R)$.  Then it is clear that $J(y_1,y_2;w) = J(y_1^+, y_2^+;w)$.  As $V_0^+$ is two-dimensional, denote by $v_1^+$, $v_2^+$ an orthonormal basis.  Thus $(y_1^+,y_2^+) = (v_1^+,v_2^+)m$ for a unique $m \in \GL_2(\R)$, and we'd like to compute $J'(m;w) := J((v_1^+,v_2^+)m;w)$ as a function of $m$.

We have $J'(m;w)$
\begin{align}
\nonumber &= \int_{M_{2}(\R)}{A_{w}((m_{11} v_1^+ + m_{21} v_2^+ + r_{11} b_{-1} + r_{21} b_{-2}) \wedge (m_{12} v_1^+ + m_{22}v_2^+ + r_{12} b_{-1} + r_{22} b_{-2}))\psi^{-1}(\tr(r))\,dr} \\ \label{eqn:Arm1} &= \frac{\det(m)^{w}}{|\det(m)|^{2w+1}} |\det(m)|^{2} \int_{M_{2}(\R)}{A_{w}((v_1^+ + r_{11} b_{-1} + r_{21} b_{-2}) \wedge (v_2^+ + r_{12} b_{-1} + r_{22} b_{-2}))\psi^{-1}(\tr(rm))\,dr}.
\end{align}
Here we have made the variable change $r \mapsto rm$ and the $|\det(m)|^{2}$ comes from the Jacobian for this change of variables.  Note that from the final expression one obtains that $J'(m;w)$ is not identically $0$ as a function of $m$ for $w >>0$, because the last integral is a Fourier transform integral.

Set $M = \diag(m,1,\,^tm^{-1}) \in \SO(V)$. For ease of notation, denote 
\[n(r) = n\left(v_1^+(r_{11} b_{-1} + r_{21}b_{-2}) + v_2^+(r_{12} b_{-1} + r_{22}b_{-2})\right)\]
so that $v_1^+ n(r) = v_1^+ + r_{11}b_{-1} + r_{21}b_{-2}$ and $v_2^+ n(r) = v_2^+ + r_{12} b_{-1} + r_{22}b_{-2}$.  Then 
\begin{align}
\nonumber \int_{M_2(\R)}{A_w(v_1^+ \wedge v_2^+ n(r)M)\psi^{-1}(\tr(r))\,dr} &= \int_{M_2(\R)}{A_w(v_1^+ \wedge v_2^+ n(m^{-1}r))\psi^{-1}(\tr(r))\,dr} \\ \label{eqn:Arm2} &= |\det(m)|^{2} \int_{M_2(\R)}{A_w(v_1^+ \wedge v_2^+ n(r))\psi^{-1}(\tr(mr))\,dr}.
\end{align}
Here the first equality is because 
\[(b_{-1},b_{-2}) m^{-1} r = \left(\,^tm^{-1} \left(\begin{array}{c} b_{-1}\\ b_{-2} \end{array}\right)\right)^t r.\]
Combining \eqref{eqn:Arm1} and \eqref{eqn:Arm2}, we obtain
\[
J'(m;w) = \frac{\det(m)^w}{|\det(m)|^{2w+1}} \int_{M_2(\R)}{A_w(v_1^+ \wedge v_2^+ n(r)M)\psi^{-1}(\tr(r))\,dr}.
\]

For $y_1, y_2 \in V_0$, set $\omega_{y_1,y_2} = e \otimes y_2 - f\otimes y_1 \in W_J$ in the notation of \cite[section A.2]{pollackQDS}.  By the multiplicity one theorem, again because $A_w(Xg)$ is a quaternionic function of $g$, the final integral is
\begin{align*}
J'(m;w) &= C_{v_1^+,v_2^+,w} \frac{\det(m)^{w}}{|\det(m)|^{2w+1}} W_{\chi_{v_1^+,v_2^+}}(M) \\ &= C_{v_1^+,v_2^+,w}  \sum_{-w \leq v \leq w}{\left(\frac{|\langle \omega_{v_1^+,v_2^+}, M r_0(i)\rangle|}{\langle \omega_{v_1^+,v_2^+}, M r_0(i)\rangle}\right)^{v} K_v(|\langle \omega_{v_1^+,v_2^+}, M r_0(i)\rangle|) \frac{x^{w+v}y^{w-v}}{(w+v)!(w-v)!}}
\end{align*}
where we have canceled the $\det(m)$'s because $W_{\chi}(M)$ contains a factor $\det(m)^w |\det(m)|$.  See \cite[page 2]{pollackE8} for any unexplained notation.  One has $\omega_{v_1^+,v_2^+} \nu(M) M^{-1} = \omega_{y_1^+,y_2^+}$ and moreover $\langle \omega_{y_1^+,y_2^+}, r_0(i) \rangle = \langle \omega_{y_1,y_2},r_0(i) \rangle$.  Consequently, $J'(m;w) = C_{v_1^+,v_2^+} W_{\chi_{y_1,y_2}}(1)$.

Putting it all together, we have proved the following.  Set
\[J(y_1,y_2,g;w) =  \int_{M_{2}(\R)}{A_{w}\left(\left((y_1 + r_{11} b_{-1} + r_{12} b_{-2}) \wedge (y_2 + r_{21} b_{-1} + r_{22} b_{-2})\right)g\right)\psi^{-1}(\tr(r))\,dr}.
\]

\begin{proposition}\label{prop:JCW} For $w >>0$, there is a nonzero constant $C_{v_1^+,v_2^+,w}$ (independent of $y_1, y_2, g$) so that $J(y_1,y_2,g;w) = C_{v_1^+,v_2^+,w} W_{\chi_{y_1,y_2}}(g)$.\end{proposition}

\subsection{Finiteness lemma} We require the following lemma, which will be used below.

\begin{lemma} Suppose $\phi \in S(\X(\A))$, and $\mu_T(h)$ is as defined in subsection \ref{subsec:Poincare}. Then for $w >>0$ the sum
\[\sum_{y_1, y_2 \in V(\Q)}\int_{N_Y(\Q)\backslash \Sp(W)(\A)}{|\mu_T(h)||\omega_{\psi}(g,h)\phi(y_1e_1+y_2e_2)|\,dh}\]
is finite.  In particular, $w > \binom{\dim(V)}{2}$ suffices.\end{lemma}
\begin{proof} The quantity of the lemma is bounded by a constant times
\[\sum_{y_1, y_2 \in \Lambda'} ||p_K(y_1\wedge y_2)||^w \int_{\GL_2(\R)}{|\det(m)|^{2w+1} e^{-2\pi(T + R(y_1,y_2),m \,^tm)}\,dm}\]
for some lattice $\Lambda'$ in $V(\R)$.  This follows from the same manipulations as in the proof of Proposition \ref{prop:Awint}.  Thus we must check the convergence of
\[\sum_{y_1, y_2 \in \Lambda'}{ \frac{||p_K(y_1\wedge y_2)||^w}{|\det(T + R(y_1,y_2))|^{w+1/2}}}.\]
Note that $T$ is fixed, and both $T$ and $R(y_1,y_2)$ are positive definite.  Moreover, $||p_K(y_1\wedge y_2)|| \leq C_0 \det(R(y_1,y_2))^{1/2}$ for some constant $C_0$. The lemma now follows easily.\end{proof}

\section{Modular forms} In this section we put together the results of the previous sections to obtain the global theorems, as stated in the introduction.

We have the following proposition.
\begin{proposition}\label{prop:thetaNice} Suppose $f$ is an automorphic form on $\Sp(W)$ corresponding to a holomorphic Siegel modular form of weight $N = w+2-n/2$, with $w >>0$. Suppose $\phi' \in S(\X(\A_f))$, set $\phi = F(\phi') \in S(\X_U(\A_f))$ the partial Fourier transform, and denote $\theta(f,\phi' \otimes \phi_\infty)$ the theta-lift of $\overline{f}$ using $\phi'$ and the special archimedean test function $\phi_\infty$. Then for $g \in \SO(V)(\R)$ we have the following Fourier expansion:
\[
 \theta(f, \phi' \otimes \phi_\infty)_Z(g) =\phi(f,\phi'\otimes \phi_\infty)_{deg}(g) + C_{v_1^+,v_2^+,w} \sum_{y_1,y_2 \in V_0(\Q): S(y_1,y_2) > 0}{a(y_1,y_2;\phi) W_{\chi_{-y_1,-y_2}}(g)}
\]
with notation as follows: $\theta(f,\phi'\otimes \phi_\infty)_{deg}$ denotes the sum of all the degenerate terms in the Fourier expansion of $\theta(f, \phi' \otimes \phi_\infty)_Z$, i.e., the sum of the Fourier terms $\theta(f, \phi' \otimes \phi_\infty)_{\chi_{y_1,y_2}}$ with $y_1, y_2 \in V_0$ with $\dim \Span\{y_1, y_2 \} \leq 1$; and
\begin{equation}\label{eqn:FCfte}
a(y_1,y_2;\phi) = \int_{M_Y(\A_f)}{\delta_P(h)^{-1}|\det(h)|_X^{n/2}\phi^K((b_1f_1 + b_2 f_2 + y_1 e_1 + y_2 e_2)h)\overline{a_{f}(S(y_1,y_2))(h)}\,dh}
\end{equation}
where $\phi^K(x) = \int_{\Sp_4(\widehat{\Z})}{\omega_{\psi,\Y_U}(1,k) \phi(x)\,dk}$.
\end{proposition}
In the statement of the proposition, the function $a_{f}(T)(h)$, for $h_f \in M_Y(\A_f)$ is defined by the equality
\[f(h_f h_\infty) = \sum_{T > 0}{a_f(T)(h_f) j(h_\infty,i)^{-N} e^{2 \pi i (T, h_\infty \cdot i)}}.\]
\begin{proof} First, by Proposition \ref{prop:PFglobal}, we can compute the theta-lift using $F(\phi' \otimes \phi_\infty)$.  Then from Proposition \ref{prop:FC1} we have
\[\theta(f,\phi \otimes F(\phi_\infty))_{-y_1,-y_2}(1) = \int_{N_Y(\A)\backslash \Sp(W)(\A)}{ \omega_{\psi,\Y_U}(1,h)\phi(z_{y_1,y_2}) \overline{f}_{-S(y_1,y_2)}(h)\,dh}.\]

This integral factors into an archimedean and finite adelic part.  Applying the Iwasawa decomposition, the finite adelic part gives the right-hand side of \eqref{eqn:FCfte}.  By Corollary \ref{cor:Kequiv} and the definition of the action $\omega_{\psi,\Y_U}$ on $S(\X_U)$, one has $\omega_{\psi,Y_U}(1,k)F(\phi_\infty) = j(k,i)^{-N} F(\phi_\infty)$ for $k  \in K_{\Sp(W),\infty} \simeq U(2)$.  Applying Iwasawa again, the archimedean part gives
\[D(y_1,y_2) := \int_{M_Y(\R)}{\delta_P(h)^{-1} \mu_{-S(y_1,y_2)}(h) \omega_{\psi,\Y_U}(1,h)F(\phi_\infty)(z_{y_1,y_2})\,dh}.\]

By Lemma \ref{lem:PFm}, this is
\begin{align*}
D(y_1,y_2) &= \int_{M_Y(\R)}\delta_P(h)^{-1} \mu_{-S(y_1,y_2)}(h) |\det(h)|_X^{n/2+2} \\ & \;\;\;\; \times \left(\int_{X^2(\R)}\psi(\langle z_1, f_1\rangle + \langle z_2, f_2 \rangle)\phi_\infty((y_1e_1+y_2e_2+b_{-1}z_1 + b_{-2}z_2)h)\,dz_1\,dz_2\right)\,dh.
\end{align*}
Now, because $||y+b||^2 = ||y||^2 + ||b||^2$ for $y \in V_0$ and $b \in U^\vee$, it is easy to see that the double integral $D(y_1,y_2)$ converges absolutely.  Thus changing the order of integration, one obtains
\[D(y_1,y_2) = \int_{M_2(\R)}{A_{w}((y_1+r_{11}b_{-1} + r_{12} b_{-2}) \wedge (y_2 + r_{21}b_{-1}+r_{22}b_{-2}))\psi(\tr(r))\,dr}\]
by Proposition \ref{prop:Awint}. But now, by Proposition \ref{prop:JCW}, this is $C_{v_1^+,v_2^+,w} W_{\chi_{-y_1,-y_2}}(1)$.  The proposition follows because the theta-lift is a quaternionic modular form.\end{proof}

\subsection{Level one} Putting everything together, we have the following theorem which describes the Fourier coefficients more precisely in the level one setting.

\begin{theorem}\label{thm:level1} Suppose that $F(Z) = \sum_{T > 0}{a_F(T) e^{2\pi i (T,Z)}}$ is a classical level one cuspidal Siegel modular form on $\Sp_4$ of weight $N=w+2-n/2$, and assume that $w >>0$ is even.  Denote by $f$ the automorphic function on $\Sp_4(\A)$ associated to $F$ so that $f(g) = j(g,i)^{-N}F(g \cdot i)$ if $g \in \Sp_4(\R)$ and let $\theta(f)$ be the theta-lift of $\overline{f}$ with all unramified data.  That is $\theta(f) = \theta(f;\phi)$ where $\phi = \phi_f \otimes \phi_\infty$ and $\phi_f$ is the characteristic function of $\widehat{\Z} \otimes \Lambda^2$ for the even unimodular lattice $\Lambda = U(\Z) \oplus V_0(\Z) \oplus U^\vee(\Z)$.  Then for $y_1, y_2 \in V_0(\Z)$ with $\dim\Span\{y_1,y_2\} =2$, one has that the Fourier coefficient $a_{\theta(f)}(y_1,y_2)$ is given by
\[
a_{\theta(f)}(y_1,y_2) = \sum_{r \in \GL_2(\Z)\backslash (\GL_2(\Q) \cap M_2(\Z)) \text{ s.t. } (y_1,y_2)r^{-1} \in V_0(\Z)^2}{|\det(r)|^{w-1} a_F( \,^tr^{-1} S(y_1,y_2) r^{-1})^*}.
\]
\end{theorem}
\begin{proof}\label{thm:F} Suppose $F$ is as in the statement of the theorem, and $h \in M_Y(\A_f)$.  We can write $h = h_\Q h_\infty^{-1} k$ for $h \in M_Y(\A_f)$, $h_\Q \in M_Y(\Q)$, $h_\infty \in M_Y(\R)$ the archimedean part of $h_\Q$ and $k \in M_Y(\widehat{\Z})$.  Then 
\[a_F(T)(h) = a_F(T)(h_\Q h_\infty^{-1} k) = a_F(T \cdot h_\Q) \det(h_\infty)_X^{-N}.\]
Plugging this in to \eqref{eqn:FCfte}, one obtains
\[a(y_1,y_2) = \int_{M_Y(\A_f)}{|\det(h)_X|^{n/2-3+N}\phi((b_1f_1 + b_2f_2 + y_1 e_1 + y_2e_2)h)a_F(S(y_1,y_2) \cdot h_\Q)^*\,dh}.\]
Note that $n/2-3+N = w-1$. Set $h = \diag(r^{-1},r^t)$.  Because $\phi_f$ is the characteristic function of the lattice $\left(U(\Z) \otimes W(\Z) \oplus V_0(\Z) \otimes X(\Z)\right) \otimes \widehat{\Z}$, $\phi((b_1 f_1 + b_2 f_2)h) \neq 0$ if and only if $r \in M_2(\widehat{\Z})$, and then $\phi((y_1e_1 + y_2 e_2)h) \neq 0$ if and only if $(y_1, y_2) r^{-1} \in V_0(\widehat{\Z})$.  The theorem follows. \end{proof}

\subsection{Modular forms on $G_2$} Assume now that $V$ is $8$-dimensional.  By Rallis \cite[Chapter I, section 3]{rallisHD}, the theta-lift $\theta(f)$ is a cusp form on $\SO(V)$.  Consequently, the degenerate terms $\theta(f;\phi'\otimes \phi_\infty)_{deg}$ in Proposition \ref{prop:thetaNice} are $0$ in this case.  Consequently, we may normalize the $\theta$-lift by dividing by $C_{v_1^+,v_2^+,w}$, and see that the Fourier coefficients $a(y_1,y_2;\phi)$ are algebraic numbers if the Fourier coefficients of $f$ are.

In this case that $V$ is $8$-dimensional, it turns out that restricting cuspidal modular forms from $\SO(V)$ to $G_2$ again produces a \emph{cuspidal} modular form. This is proven in Corollary \ref{cor:restrict} below, of which the main step is the following lemma.

\begin{lemma}\label{lem:push} Denote by $E = \R \times \R \times \R$, and suppose that $v=(a,b,c,d) \in W_{E} = \R \oplus E \oplus E \oplus \R$ is positive definite, i.e., $v >0$.  Let $v' =  (a',b',c',d'):=(a,\tr(b)/3,\tr(c)/3,d) \in W_{\R}$ be the image of the $\GL_2$-equivariant projection $W_{E} \rightarrow W_{\R}$.  Then $v' > 0$.\end{lemma}
\begin{proof} Recall that $v > 0$ means that the three binary quadratic forms associated to $v$ are each positive-definite.  These three binary quadratic forms we write as 
\[S(v) = \mb{b^\#-ac}{ad-bc-\tr(ad-bc)/2}{ad-bc-\tr(ad-bc)/2}{c^\#-db}.\]
See \cite[Example 4.4.2]{pollackLL}.

Now, note that if $p(x,y)$ is a real binary cubic form, then a $\GL_2$-translate of it is divisible by $x$.  Indeed, this follows from the fact that every real cubic polynomial has a real root.  Thus, by $\GL_2$ equivariance, we may assume $d=0$.  Now, consider $S(v) = \mb{b^\#-ac}{*}{*}{c^\#}$.  Because we assume $v > 0$, we obtain that $c^\#$ is positive definite.  It follows that $c$ is either positive-definite or negative-definite.  By multiplying by $-1$ if necessary, we can assume $c > 0$.  

Now, because $c > 0$, $\tr(c) > 0$ and thus by acting by elements $\overline{n}(\R)$, we can assume $\tr(b) = 0$.  But if $\tr(b) = 0$, then either $b=0$ and $b^\# = 0$, or one component of $b^\#$ is negative.  Because $v > 0$, $b^\# - ac > 0$ and thus since $c > 0$ we must have $a < 0$.  Hence $v' = (a,0,c',0)$ with $a < 0$ and $c' > 0$.  Therefore $v' > 0$, as desired.\end{proof}

\begin{corollary}\label{cor:restrict} Suppose that $f$ is a cuspidal modular form on $\SO(4,4)$ of weight $w$, and denote by $f'$ the restriction of $f$ to $G_2$.  Then $f'$ is a cuspidal modular form on $G_2$ of weight $w$.\end{corollary}
\begin{proof} By, for example, \cite[Theorem 7.3.1]{pollackQDS}, one can check by hand that $f'$ is a modular form of weight $w$.  Now, the map $G_2 \rightarrow \SO(V) = \SO(4,4)$ factors through $\Spin(8)$.  Or relatedly, the image of the real points $G_2(\R)$ sits in the connected component of the identity of $\SO(4,4)(\R)$, as $G_2(\R)$ is connected.  Because of either of these facts, one sees that the only Fourier coefficients of $f$ that contribute to $f'$ are those that correspond to $v \in W_{E}$ with $S(v) > 0$.  That is, all three binary quadratic forms associated to $v$ must be positive definite.  From Lemma \ref{lem:push}, one sees that $f'$ only has nonzero Fourier coefficients associated to $v' \in W_\R$ with $v'> 0$.  Consequently, $f'$ is cuspidal, as desired.\end{proof}

In case $F(Z)$ is a level one Siegel modular form of weight $w >>0$ as in Theorem \ref{thm:F}, one obtains the following corollary.
\begin{corollary}\label{cor:G2level1} Suppose $F(Z) = \sum_{T > 0}{a_F(T)q^T}$ is a level one Siegel modular form on $\Sp(4)$ of sufficiently large even weight $w$.  For $p(x,y) = ax^3 + bx^2 y + cxy^2 + dy^3$ an integral binary cubic form that factors into three distinct linear factors over $\R$, define
\[
a_{\theta(f)}(p) := \sum_{\widetilde{p}} \sum_{r \in \GL_2(\Z)\backslash (\GL_2(\Q) \cap M_2(\Z)): \widetilde{p}r^{-1} \text{ integral}} |\det(r)|^{w-1} a_F(\,^tr^{-1} S(\widetilde{p}) r^{-1})^*
\]
where the sum is over integral boxes $\widetilde{p} = (a,(b_1,b_2,b_3),(c_1,c_2,c_3),d)$ with $b_1 + b_2 + b_3 = b$ and $c_1 + c_2 + c_3 = c$.  Then there is a cuspidal modular form $\theta(f)|_{G_2}$ on $G_2$ with $p^{th}$ Fourier coefficient $a_{\theta(f)}(p)$. Moreover, the Fourier coefficient of $\theta(f)|_{G_2}$ corresponding to $p(x,y) = x^3 - xy^2$ is $a_F(\mm{1}{0}{0}{1})$. \end{corollary}
\begin{proof} The only thing that remains to be proved is the final statement.  For this, suppose that $v=(a,(b_1,b_2,b_3),(c_1,c_2,c_3),d)$ is a $2 \times 2 \times 2$ integer box with $v > 0$, and $v' = (a,\tr(b)/3,\tr(c)/3),d) = (-1,0,1,0)$.  Then $a=-1$ and $d=0$.  Because $c$ is positive-definite and integral with $\tr(c) = 3$, we must have $c= (1,1,1)$.  Moreover, because $\tr(b)=0$, $b$ is integral, and $b^\# + c > 0$, one must have $b=0$.  Thus the only box $v$ lying above $(-1,0,1,0)$ is $\widetilde{p}=(-1,(0,0,0),(1,1,1),0)$.  Because $S(\widetilde{p}) = \mm{1}{0}{0}{1}$, the corollary follows.\end{proof}

\bibliographystyle{amsalpha}
\bibliography{QDS_Bib} 

\providecommand{\bysame}{\leavevmode\hbox to3em{\hrulefill}\thinspace}
\providecommand{\MR}{\relax\ifhmode\unskip\space\fi MR }
\providecommand{\MRhref}[2]{%
  \href{http://www.ams.org/mathscinet-getitem?mr=#1}{#2}
}
\providecommand{\href}[2]{#2}
\begin{thebibliography}{GGS02}

\bibitem[DN70]{doiNaganuma}
Koji Doi and Hidehisa Naganuma, \emph{On the functional equation of certain
  {D}irichlet series}, Invent. Math. \textbf{9} (1969/1970), 1--14.
  \MR{0253990}

\bibitem[Gan00]{ganSW}
Wee~Teck Gan, \emph{A {S}iegel-{W}eil formula for exceptional groups}, J. Reine
  Angew. Math. \textbf{528} (2000), 149--181. \MR{1801660}

\bibitem[GGS02]{ganGrossSavin}
Wee~Teck Gan, Benedict Gross, and Gordan Savin, \emph{Fourier coefficients of
  modular forms on {$G_2$}}, Duke Math. J. \textbf{115} (2002), no.~1,
  105--169. \MR{1932327}

\bibitem[GW94]{grossWallach1}
Benedict~H. Gross and Nolan~R. Wallach, \emph{A distinguished family of unitary
  representations for the exceptional groups of real rank {$=4$}}, Lie theory
  and geometry, Progr. Math., vol. 123, Birkh\"auser Boston, Boston, MA, 1994,
  pp.~289--304. \MR{1327538}

\bibitem[GW96]{grossWallach2}
\bysame, \emph{On quaternionic discrete series representations, and their
  continuations}, J. Reine Angew. Math. \textbf{481} (1996), 73--123.
  \MR{1421947}

\bibitem[Kli90]{klingen}
Helmut Klingen, \emph{Introductory lectures on {S}iegel modular forms},
  Cambridge Studies in Advanced Mathematics, vol.~20, Cambridge University
  Press, Cambridge, 1990. \MR{1046630}

\bibitem[Kud78]{kudla}
Stephen~S. Kudla, \emph{Theta-functions and {H}ilbert modular forms}, Nagoya
  Math. J. \textbf{69} (1978), 97--106. \MR{0466025}

\bibitem[Kud86]{kudlaInv}
\bysame, \emph{On the local theta-correspondence}, Invent. Math. \textbf{83}
  (1986), no.~2, 229--255. \MR{818351}

\bibitem[KV78]{kV}
M.~Kashiwara and M.~Vergne, \emph{On the {S}egal-{S}hale-{W}eil representations
  and harmonic polynomials}, Invent. Math. \textbf{44} (1978), no.~1, 1--47.
  \MR{0463359}

\bibitem[LS93]{liSchwermer}
Jian-Shu Li and Joachim Schwermer, \emph{Constructions of automorphic forms and
  related cohomology classes for arithmetic subgroups of {$G_2$}}, Compositio
  Math. \textbf{87} (1993), no.~1, 45--78. \MR{1219452}

\bibitem[LV80]{lionVergne}
G\'{e}rard Lion and Mich\`ele Vergne, \emph{The {W}eil representation, {M}aslov
  index and theta series}, Progress in Mathematics, vol.~6, Birkh\"{a}user,
  Boston, Mass., 1980. \MR{573448}

\bibitem[Nar08]{narita1}
Hiro-aki Narita, \emph{Theta lifting from elliptic cusp forms to automorphic
  forms on {${\rm Sp}(1,q)$}}, Math. Z. \textbf{259} (2008), no.~3, 591--615.
  \MR{2395128}

\bibitem[Niw75]{niwa}
Shinji Niwa, \emph{Modular forms of half integral weight and the integral of
  certain theta-functions}, Nagoya Math. J. \textbf{56} (1975), 147--161.
  \MR{0364106}

\bibitem[Oda78]{oda}
Takayuki Oda, \emph{On modular forms associated with indefinite quadratic forms
  of signature {$(2, n-2)$}}, Math. Ann. \textbf{231} (1977/78), no.~2,
  97--144. \MR{0466026}

\bibitem[Pol18a]{pollackQDS}
Aaron Pollack, \emph{The {F}ourier expansion of modular forms on quaternionic
  exceptional groups}.

\bibitem[Pol18b]{pollackLL}
\bysame, \emph{Lifting laws and arithmetic invariant theory}, Camb. J. Math
  \textbf{6} (2018), no.~4, 347--449.

\bibitem[Pol18c]{pollackE8}
\bysame, \emph{The minimal modular form on quaternionic ${E}_8$}.

\bibitem[Pol18d]{pollackG2}
\bysame, \emph{Modular forms on ${G}_2$ and their standard ${L}$-function}.

\bibitem[PS83]{PSsk}
I.~I. Piatetski-Shapiro, \emph{On the {S}aito-{K}urokawa lifting}, Invent.
  Math. \textbf{71} (1983), no.~2, 309--338. \MR{689647}

\bibitem[Ral82]{rallisLF}
Stephen Rallis, \emph{Langlands' functoriality and the {W}eil representation},
  Amer. J. Math. \textbf{104} (1982), no.~3, 469--515. \MR{658543}

\bibitem[Ral84]{rallisHD}
S.~Rallis, \emph{On the {H}owe duality conjecture}, Compositio Math.
  \textbf{51} (1984), no.~3, 333--399. \MR{743016}

\bibitem[RS78]{RS78}
S.~Rallis and G.~Schiffmann, \emph{Automorphic forms constructed from the
  {W}eil representation: holomorphic case}, Amer. J. Math. \textbf{100} (1978),
  no.~5, 1049--1122. \MR{517145}

\bibitem[RS81]{RS81}
\bysame, \emph{On a relation between {$\widetilde {\rm SL}_{2}$} cusp forms and
  cusp forms on tube domains associated to orthogonal groups}, Trans. Amer.
  Math. Soc. \textbf{263} (1981), no.~1, 1--58. \MR{590410}

\bibitem[RS89]{RSG2}
\bysame, \emph{Theta correspondence associated to {$G_2$}}, Amer. J. Math.
  \textbf{111} (1989), no.~5, 801--849. \MR{1020830}

\bibitem[Shi75]{shintani}
Takuro Shintani, \emph{On construction of holomorphic cusp forms of half
  integral weight}, Nagoya Math. J. \textbf{58} (1975), 83--126. \MR{0389772}

\bibitem[Wal03]{wallach}
Nolan~R. Wallach, \emph{Generalized {W}hittaker vectors for holomorphic and
  quaternionic representations}, Comment. Math. Helv. \textbf{78} (2003),
  no.~2, 266--307. \MR{1988198}

\bibitem[Wei64]{weil}
Andr\'{e} Weil, \emph{Sur certains groupes d'op\'{e}rateurs unitaires}, Acta
  Math. \textbf{111} (1964), 143--211. \MR{0165033}

\bibitem[Wei06]{weissman}
Martin~H. Weissman, \emph{{$D_4$} modular forms}, Amer. J. Math. \textbf{128}
  (2006), no.~4, 849--898. \MR{2251588}

\bibitem[YN12]{naritaYamauchi}
Atsuo Yamauchi and Hiro-Aki Narita, \emph{Some vector-valued singular
  automorphic forms on {$U(2,2)$} and their restriction to {${\rm Sp}(1,1)$}},
  Internat. J. Math. \textbf{23} (2012), no.~10, 1250104, 27. \MR{2999049}

\end{thebibliography}

\end{document}